\documentclass{amsart}
\usepackage{amssymb,graphicx,mathtools,amsmath,float}
\usepackage{calc,tipa,pgfplots}
\usepackage{tikz, adjustbox}
\usepackage{subcaption}
\newtheorem{theorem}{Theorem}[section]
\newtheorem{lemma}[theorem]{Lemma}
\newtheorem{prop}[theorem]{Proposition}
\newtheorem{conj}[theorem]{Conjecture} 
\newtheorem{cor}[theorem]{Corollary}
\newtheorem{que}[theorem]{Question}

\theoremstyle{definition}

\theoremstyle{remark}
\newtheorem{remark}[theorem]{Remark}

\numberwithin{equation}{section}

\newcommand{\RNum}[1]{\uppercase\expandafter{\romannumeral #1\relax}}
\newlength{\mywidth}

\begin{document}

\title[the Fraser-Li conjecture and the Weierstrass representation]{A new approach to the Fraser-Li conjecture with the Weierstrass representation formula}

\author[J. Lee]{Jaehoon Lee}
\address[]{Jaehoon Lee, Department of Mathematical Sciences, Seoul National University, Seoul  08826, Korea}
\email{jaehoon.lee@snu.ac.kr}

\author[E. Yeon]{Eungbeom Yeon}
\address[]{Eungbeom Yeon, Department of Mathematical Sciences, Seoul National University, Seoul  08826, Korea}
\email{y2bum@snu.ac.kr}


\date{}

\commby{Jiaping Wang}

\begin{abstract}
In this paper, we provide a sufficient condition for a curve on a surface in $\mathbb{R}^3$ to be given by an orthogonal intersection with a sphere. This result makes it possible to express the boundary condition entirely in terms of the Weierstrass data without integration when dealing with free boundary minimal surfaces in a ball $\mathbb{B}^3$. Moreover, we show that the Gauss map of an embedded free boundary minimal annulus is one to one. By using this, the Fraser-Li conjecture can be translated into the problem of determining the Gauss map. On the other hand, we show that the Liouville type boundary value problem in an annulus gives some new insight into the structure of immersed minimal annuli orthogonal to spheres. It also suggests a new PDE theoretic approach to the Fraser-Li conjecture.
\end{abstract}

\maketitle

\section{Introduction}\label{intro}
The theory of free boundary minimal surfaces has been a very active field of research. One of the most widely accepted conjectures is the following: 
	\begin{conj}[Fraser and Li, \cite{FraserLi}]\label{FL}
	The critical catenoid is the only embedded free boundary minimal annulus in $\mathbb{B}^3$, up to rigid motions.
	\end{conj}
The above open question deals with the free boundary analog of the Lawson conjecture in which the Clifford torus is the only embedded minimal torus in $\mathbb{S}^3$. Lawson's conjecture was proved in \cite{Brendle} by applying a maximum principle to a two-point function obtained from the geometric observation of the inner and outer spheres of a surface in $\mathbb{S}^3$. It is tempting to check whether a similar method holds, but two boundary components of the surface make it difficult to apply a maximum principle type method to the Fraser-Li conjecture.

Instead, another useful tool called the \emph{Weierstrass representation formula} is used in the classical minimal surface theory:
\begin{align*}
\mbox{Re}\int\left[\frac{1}{2}(1-g^2)\omega, \frac{i}{2}(1+g^2)\omega, g\omega\right],
\end{align*}
where $g$ is a meromorphic function and $\omega$ is a holomorphic one-form on a Riemann surface. Since the Weierstrass representation formula is presented in the integral form, it is difficult to translate all the information related to free boundary into the data $g$ and $\omega$.

Meanwhile, two facts on boundary curves can be obtained from the free boundary condition. More generally, let $\Sigma$ be a surface in $\mathbb{R}^3$ that meets a sphere orthogonally along a curve $\Gamma$. As implied by the Terquem-Joachimsthal theorem \cite{Spi}, $\Gamma$ is a curvature line on $\Sigma$. Moreover, as the conormal vector of $\Sigma$ along the curve coincides with the unit normal vector to the sphere, the geodesic curvature along $\Gamma$ computed on $\Sigma$ is identical with the normal curvature on the sphere equal to 1. It turns out that the converse is also true. Indeed, we gained the below observation as the two conditions are so powerful:
\begin{prop}[Proposition \ref{Thm1} in Section \ref{Main1}]
Let $\Gamma$ be a compact connected real analytic curve on a surface $\Sigma$ in $\mathbb{R}^3$. Assume further that $\Gamma$ is a line of curvature along which the principal curvature vanishes only at finitely many points. If $\Gamma$ has constant geodesic curvature $c$ on $\Sigma$, then there exists a sphere $S$ of radius $\frac{1}{|c|}$ (if $c=0$, it means a plane) where $\Sigma$ intersects $S$ orthogonally along $\Gamma$. When $\Gamma$ is a piecewise real analytic curve, the same result can be obtained once a sphere is replaced by a union of spheres.
\end{prop}
This proposition generalizes the well-known fact that if a surface contains a principal geodesic, then it meets the plane containing the geodesic perpendicularly. 

If $\Sigma$ is a minimal surface, then the condition on the principal curvature can be equivalently stated that $\Gamma$ contains only a finite number of umbilic points. Since curvature lines on a nonplanar minimal surface are real analytic and contain possibly a finite number of umbilic points, the proposition is applied to both a curve on the interior of a minimal surface and real analytic boundaries. It should be noted that the conditions in the proposition can be expressed in terms of the Weierstrass data $(g, \omega)$ without undergoing a process of integration. In this way, the proposition solves difficulties associated with using the representation formula for free boundary minimal surfaces in a ball.

In Section \ref{gauss} we prove that the Gauss map of a free boundary minimal annulus in $\mathbb{B}^3$ with embedded boundary is a one to one map. This enables us to globally parametrize the surface via the Gauss map. Analyzing the Hopf differential, the Fraser-Li conjecture eventually becomes the problem of determining holomorphic functions on concentric annuli with boundary conditions from the geodesic curvature. As a conformal diffeomorphism from a doubly-connected region to a concentric annulus is uniquely determined up to scaling and rotations, the conjecture becomes the problem on the shape of the Gauss image. However, the boundary conditions depend not only on the geometric term (curvature of the boundaries of the Gauss image) but also on the parametrization itself. This fact is discussed in Remark \ref{BC}. It remains the main difficulty with this approach.

On the other hand, we observe that the Liouville equation gives some new insight into the conjecture. More specifically, the Liouville type boundary value problem in an annulus (E[$R$, $\epsilon$, $C_0$] in Section \ref{Five}) can be obtained from the Simons identity on a free boundary minimal annulus. Conversely, we construct a minimal immersion from the Liouville solution by using the above proposition and the Weierstrass representation formula in Theorem \ref{Liou}. This minimal immersion is defined on the universal cover of the annulus and orthogonal to unit spheres. As a corollary, it follows that an immersed free boundary minimal annulus in a ball can be divided into congruent pieces, which we call a fundamental piece. Studying the geometry of fundamental pieces will be an interesting future direction. See Section \ref{Five} for more details.

It should also be mentioned that Jim\'{e}nez \cite{Jimenez} solved the Liouville equation in an annulus to classify constant curvature annuli. Although it was successful to deal with the boundary in \cite{Jimenez} by considering the Schwarzian derivative of a meromorphic function (for instance, $h$ in (\ref{general})), the free boundary condition does not work well with this method. Hence only using the holomorphic function theory is a difficult approach. In this regard, considering the Liouville equation suggests a new PDE theoretic approach to the Fraser-Li conjecture. 

The paper is organized as follows. In Section \ref{prelim} the Weierstrass representation and the well-known fact on a characterization of spherical space curves are reviewed. In Section \ref{Main1} we prove that a necessary condition obtained from the orthogonal intersection with a sphere also become a sufficient condition for the existence of a sphere that meets the surface orthogonally. Then we show in Section \ref{gauss} that the Gauss map is one to one on an embedded free boundary minimal annulus in a ball. In the final section, the relation between the Fraser-Li conjecture and the Liouville equation is addressed.

\section{Preliminaries}\label{prelim}
\subsection{Weierstrass representation}
We recall the Weierstrass representation of a minimal surface in $\mathbb{R}^3$. Since coordinate functions are harmonic, it can be expressed as a conformal harmonic immersion from a Riemann surface $C$ into $\mathbb{R}^3$: 
	\begin{align*}
	\mbox{Re}\int\left[\frac{1}{2}(1-g^2)\omega, \frac{i}{2}(1+g^2)\omega, g\omega\right],
	\end{align*}
where $g$ is a meromorphic function and $\omega$ is a holomorphic one-form on $C$ such that $g$ has order $n$ pole at $p\in C$ if and only if $\omega$ has order $2n$ zero at $p\in C$. Note that $g$ corresponds to the Gauss map and $(g, \omega)$ is called a \emph{Weierstrass data}.

The induced metric is 
	\begin{align*}
	\textup{d} s^2=\frac{1}{4}(1+|g|^2)^2|\omega|^2
	\end{align*}
and the second fundamental form is given by
	\begin{align*}
	\mbox{Re}\{\textup{d} g\cdot\omega\}.
	\end{align*} 
	To obtain a well-defined immersion from $C$, it should satisfy the period condition:
	\begin{align*}
	\mbox{Re}\int_{\delta}\left[\frac{1}{2}(1-g^2)\omega, \frac{i}{2}(1+g^2)\omega, g\omega\right]=0
	\end{align*}
for every closed curves $\delta$ on $C$. Otherwise, the Weierstrass data only gives a well-defined minimal immersion on the universal cover of $C$.

\subsection{Characterization of spherical space curves}
We may recall an elementary fact on spherical curves, which is one of the main ingredients of the paper. Since a curve in $\mathbb{R}^3$ is determined (up to a rigid motion) by the curvature and torsion, it is possible to characterize spherical space curves in terms of them as follows.

Let $\alpha=\alpha(t): I\to \mathbb{R}^3$ be an arclength parametrized curve such that $\tau(t)\neq 0$ and $\kappa'(t)\neq 0$ for all $t\in I$. Here, $\kappa$ is the curvature and $\tau$ is the torsion. Then $\alpha(I)\subset \mathbb{S}^2(R)$ if and only if
\begin{align}\label{DoC}
	\left(\frac{1}{\kappa}\right)^2+\left(\left(\frac{1}{\kappa}\right)'\right)^2\left(\frac{1}{\tau}\right)^2=R^2.
\end{align}
Moreover, if $\alpha(I)\subset \mathbb{S}^2(R)$, then the unit normal vector of the sphere can be expressed as
	\begin{align}\label{DoCC}
	-\frac{1}{R\kappa}n+\frac{\kappa'}{R\kappa^2\tau}b,
	\end{align}
where $n$ and $b$ are the normal and binormal vector of $\alpha$, respectively. 

\section{Sufficient condition to meet a sphere orthogonally}\label{Main1}
In this section, we discuss a sufficient condition for a surface to meet a sphere orthogonally. We first prove two lemmas that describe the local nature of orthogonal intersections. Then by combining two lemmas and some global arguments, we prove the main result. 
\begin{lemma}\label{Lem1}
Let $\Gamma$ be a line of curvature on a surface $\Sigma$ with constant geodesic curvature $c\neq0$. Suppose that the principal curvature along $\Gamma$ does not vanish. Moreover, assume that $\Gamma$ has a non-vanishing torsion as a curve in $\mathbb{R}^3$. Then there exists a sphere of radius $\frac{1}{|c|}$ such that it intersects $\Sigma$ orthogonally along $\Gamma$.
\end{lemma}

\begin{proof}
Let $t$ be the unit tangent vector of $\Gamma$. Let us denote the unit conormal along $\Gamma$ as $\nu$ and the unit normal to $\Sigma$ as $N$ such that $\{N, t, \nu\}$ is positively oriented in $\mathbb{R}^3$. Also we write the unit normal and binormal of $\Gamma$ by $n$ and $b$, respectively. Here $b$ is given by $t\wedge n$.

As $\{N, \nu\}$ and $\{b, n\}$ form oriented orthonormal bases for the normal plane of the curve, we may write 
	\begin{equation}\label{star}
		\begin{cases}
		&b=\cos\theta N+\sin\theta \nu\\
		&n=-\sin\theta N+\cos\theta \nu
		\end{cases}
	\end{equation}
for some function $\theta$ defined on $\Gamma$.

Let $\kappa$ and $\tau$ be the curvature and torsion of $\Gamma$ as a space curve. Since $\Gamma$ is a line of curvature on $\Sigma$, 
	\begin{align}
	\langle\nabla_t N, \nu\rangle=-\langle N, \nabla_t \nu\rangle=0,
	\end{align}
and we obtain from (\ref{star}) that
	\begin{align*}
	&\langle\nabla_t N, n\rangle=\cos\theta\langle\nabla_t N, \nu\rangle=0,\\ 
	&\langle\nabla_t \nu, n\rangle=-\sin\theta\langle\nabla_t\nu, N\rangle=0.
	\end{align*}
Here $\nabla$ means the Riemannian connection in $\mathbb{R}^3$. Then it follows from the Frenet-Serret formulas and (\ref{star}) that 
	\begin{align*}
	-\tau&=\langle\nabla_tb, n\rangle\\
		&=\langle\nabla_t(\cos\theta N+\sin\theta \nu), n\rangle\\
		&=(D_t \theta)\langle-\sin\theta N+\cos\theta \nu, n\rangle+\cos\theta\langle\nabla_tN, n\rangle+\sin\theta\langle\nabla_t\nu, n\rangle\\
		&=D_t \theta,
	\end{align*}
where $D_t$ denotes the directional derivative. Hence $\tau=-D_t \theta$.

Again by (\ref{star}) and the Frenet-Serret formulas, the geodesic curvature of $\Gamma$ computed on $\Sigma$ is given by
	\begin{align}\label{const}
	c^2=\langle\nabla_tt, \nu\rangle^2=\langle\kappa n, \nu\rangle^2=\kappa^2\cos^2\theta.
	\end{align}
Since the principal curvature along $\Gamma$ is nonzero and $c\neq0$, we have
	\begin{align*}
	\kappa\neq 0,\ \cos\theta\neq 0,\ \sin\theta\neq 0.
	\end{align*}
The constancy of the geodesic curvature along $\Gamma$ implies that
	\begin{align*}
	0=D_t(\kappa\cos\theta)=(D_t\kappa)\cos\theta -\kappa\sin\theta(D_t\theta)=(D_t\kappa)\cos\theta+\tau\kappa\sin\theta,
	\end{align*}
where we used $\tau=-D_t\theta$ in the last step. Therefore we obtain 
	\begin{align}
	\frac{D_t\kappa}{\tau\kappa}=-\tan\theta\neq 0.
	\end{align}
	
Now we compute
	\begin{align*}
	\frac{1}{\kappa^2}+\left(D_t\left(\frac{1}{\kappa}\right)\right)^2\frac{1}{\tau^2}&=\frac{1}{\kappa^2}\left(1+\frac{(D_t\kappa)^2}{\tau^2\kappa^2}\right)\\
	&=\frac{1}{\kappa^2}(1+\tan^2\theta)\\
	&=\frac{1}{c^2}
	\end{align*}
so that the characterization result in Section \ref{prelim} shows that $\Gamma$ lies on a sphere of radius $\frac{1}{|c|}$. Moreover, we deduce from
	\begin{align*}
	\left\langle-\frac{1}{\kappa}n+\frac{D_t\kappa}{\kappa^2\tau}b, N\right\rangle&=-\frac{1}{\kappa}\langle n, N\rangle+\frac{1}{\kappa}\cdot\frac{D_t\kappa}{\tau\kappa}\langle b, N\rangle\\
	&=\frac{1}{\kappa}\sin\theta-\frac{1}{\kappa}\tan\theta\cdot\cos\theta\\
	&=0
	\end{align*}
that the sphere intersects $\Sigma$ orthogonally.
\end{proof}

Next we have the following lemma for a vanishing torsion case:
\begin{lemma}\label{Lem2}
Let $\Gamma$ be a line of curvature on a surface $\Sigma$ with constant geodesic curvature $c\neq0$. If the torsion of $\Gamma$ is identically zero, then $\Gamma$ is a part of a circle. Therefore $\Sigma$ is orthogonal to a sphere of radius $\frac{1}{|c|}$ along $\Gamma$.
\end{lemma}

\begin{proof}
We may use the same notation as in the proof of Lemma \ref{Lem1}. Since the torsion is identically zero, $\Gamma$ lies on a plane orthogonal to $b$. Moreover, $\tau=-D_t\theta=0$ implies that $\Sigma$ has constant contact angle $\theta$ with the plane along $\Gamma$. Now (\ref{const}) shows that $\Gamma$ has constant curvature on the plane and the result follows.
\end{proof}

Combining Lemma \ref{Lem1} and \ref{Lem2}, we obtain the proposition:
\begin{prop}\label{Thm1}
Let $\Gamma$ be a compact connected real analytic curve on a surface $\Sigma$ in $\mathbb{R}^3$. Assume further that $\Gamma$ is a line of curvature along which the principal curvature vanishes only at finitely many points. If $\Gamma$ has constant geodesic curvature $c$ on $\Sigma$, then there exists a sphere $S$ of radius $\frac{1}{|c|}$ (if $c=0$, it means a plane) where $\Sigma$ intersects $S$ orthogonally along $\Gamma$. When $\Gamma$ is a piecewise real analytic curve, the same result can be obtained once a sphere is replaced by a union of spheres.
\end{prop}

\begin{proof}
It is well known that if there exists a principal geodesic on $\Sigma$, then it is a plane curve and $\Sigma$ intersects that plane orthogonally. So we may assume that the geodesic curvature is nonzero, i.e., $c\neq0$. 

We will consider the real analytic case first. Since $\Gamma$ contains finitely many points where the principal curvature vanishes, it divides into finite pieces by those points: $\Gamma=\Gamma_1\cup\Gamma_2\cup\cdots\cup\Gamma_k$. Moreover, by the analyticity, the torsion of $\Gamma$ satisfies either one of the following: it is identically zero, or it vanishes only at finitely many points.

In the first case, we may apply Lemma \ref{Lem2} to each $\Gamma_i$ and obtain a sphere $S_i$ of radius $\frac{1}{|c|}$ for each $\Gamma_i$. Then we can conclude that all $S_i$'s must be the same by the continuity of the curvature of $\Gamma$. 

For the second case, each $\Gamma_i$ divides into finite pieces by torsion-vanishing points. Then Lemma \ref{Lem1} implies that we can find a sphere of radius $\frac{1}{|c|}$ for each piece, and again by the continuity, we complete the proof.

If $\Gamma$ is a piecewise real analytic curvature line, then it is not possible to use the continuity argument at singular points to obtain only one sphere as in the above. Instead, a similar argument shows that there exists a union of spheres that intersects $\Sigma$ orthogonally along $\Gamma$.
\end{proof}

\section{The Gauss map of a free boundary minimal annulus}\label{gauss}

In this section, we show that the Gauss map of a free boundary minimal annulus in a ball $\mathbb{B}^3$ with embedded boundary curves is necessarily a one to one map. Let $\Sigma$ be a free boundary minimal annulus represented by a minimal immersion $F: A(1, R)\to\mathbb{R}^3$ from $A(1, R)\coloneqq\{z\in\mathbb{C}\ |\ 1<|z|<R\}$.

We begin by observing the Hopf differential $\Phi(z)dz^2\coloneqq(F_{zz}\cdot N)dz^2$. Here $N$ is the unit normal to the surface $\Sigma$. As in \cite{Nitsche}, we may consider the holomorphic function
\begin{align*}
f(z)\coloneqq z^2\Phi(z).
\end{align*}
It follows from the free boundary condition that $|z|=1$ and $|z|=R$ are curvature lines on $\Sigma$. If we put $z=re^{i\theta}$, this gives
\begin{align*}
\mbox{Im}\left\{\Phi(re^{i\theta})dz^2(\tfrac{\partial}{\partial\theta}, \tfrac{\partial}{\partial\theta})\right\}=\mbox{Im}\{-f(re^{i\theta})\}=0
\end{align*}
when $r=1$ and $r=R$. Since $\mbox{Re}f$ and $\mbox{Im}f$ are conjugate harmonic functions on the annulus, we can conclude that 
\begin{align*}
\mbox{Im}f(z)\equiv0,\ \mbox{Re}f(z)\equiv C_0
\end{align*}
for some real constant $C_0$. Therefore the second fundamental form $\sigma$ is given by
\begin{align}\label{secondf}
\sigma=\mbox{Re}\left\{\Phi(z)dz^2\right\}=\mbox{Re}\left\{\tfrac{C_0}{z^2}dz^2\right\}.
\end{align}
This implies that all concentric circles centered at the origin of $A(1, R)$ become curvature lines of $\Sigma$, and there are no umbilic points on the surface up to the boundaries. Therefore the normal curvature of the boundary computed on $\Sigma$ is never zero. Consequently, it follows from the free boundary condition that the geodesic curvature of the boundary computed on the unit sphere $\partial \mathbb{B}^3$ never changes its sign. So the boundary curve is locally convex on $\partial\mathbb{B}^3$. 

To show that the Gauss map $g : A(1,R) \rightarrow \mathbb{C}$ is a one to one map, it is enough to show that it is one to one on each of the boundaries $|z|=1$ and $|z|=R$ of the domain $A(1, R)$, denoted as $\partial A_1$ and $\partial A_2$, respectively. Indeed, if the Gauss map is one to one on $\partial A_i (i=1,2)$, then we can show that the Gauss map sends $A(1, R)$ univalently onto an annular domain bounded by two closed curves $g(\partial A_1)$ and $g(\partial A_2)$ as follows. After rotating the surface if necessary, we can assume that the Gauss map omits a north pole so that $g$ can be viewed as a holomorphic function defined on $A(1, R)$. It is well known by Cauchy's integral formula that winding number $n(g(\partial A_i),a)$ tells us total number of zeroes of $g(z)-a$ enclosed by a closed curve $\partial A_i$ and can be computed as the integral $\frac{1}{2\pi i}\int_{\partial A_i} \frac{g'(z)}{g(z)-a}dz$. It follows that 
\begin{align*}
n(g(\partial A_2),a) - n(g(\partial A_1),a) = \frac{1}{2\pi i}\int_{\partial A_2} \frac{g'(z)}{g(z)-a}dz -  \frac{1}{2\pi i}\int_{\partial A_1} \frac{g'(z)}{g(z)-a}dz
\end{align*}
is the number of points whose image under $g$ become $a$ so that its value is bigger than or equal to $0$ for $a\in \mathbb{C}$. On the other hand, since $g'$ is never zero on the domain $A(1, R)$ as the surface is foliated by lines of curvature and $g$ is one to one on each of the boundaries,
\begin{align*}
n(g(\partial A_2),a) - n(g(\partial A_1),a) = \frac{1}{2\pi i}\int_{g(\partial A_2)} \frac{1}{w-a}dw -  \frac{1}{2\pi i}\int_{g(\partial A_1)} \frac{1}{w-a}dw.
\end{align*}
We now see that $g(\partial A_1)$ must lie inside of $g(\partial A_2)$ since otherwise must have a point $a$ such that the above equation has a negative value. Accordingly, we can conclude that the Gauss map sends $A(1, R)$ univalently onto the annular domain bounded by two closed curves $g(\partial A_i) (i=1,2)$.

\begin{theorem}\label{prop2}
Let $F: A(1, R) \to \mathbb{R}^3$ be a free boundary minimal annulus in a ball $\mathbb{B}^3$ whose boundary curves $F(\partial A_1)$ and $F(\partial A_2)$ are embedded, where $\partial A_1$ and $\partial A_2$ denote boundary circles $\{z\in\mathbb{C}\ |\ |z|=1\}$ and $\{z\in\mathbb{C}\ |\ |z| = R\}$, respectively. Then the Gauss map $g(z) : A(1, R) \rightarrow \mathbb{C}$ of the surface is a one to one map that sends $A(1, R)$ to the doubly-connected region bounded by $g(\partial A_1)$ and $g(\partial A_2 )$.
\end{theorem}

\begin{proof}
As discussed above, it is enough to show that the Gauss map $g$ is one to one on each $\partial A_i$ $(i=1,2)$. Assume the contrary that $g(z_1)= g(z_2)$ where $z_1$ and $z_2$ are two distinct points on $\partial A_i$. For the sake of simplicity, let us assume that they are on $\partial A_1$. The free boundary condition implies that the normal vector of the surface is perpendicular to the position vector along the boundary (see Figure 1). So we can find a great circle $C_1$ on the unit sphere $\partial \mathbb{B}^3$ such that $F(\partial A_1)$ is tangent to $C_1$ at both $F(z_1)$ and $F(z_2)$.

Let us denote $S_1$ as an open hemisphere whose boundary is the great circle $C_1$. Near two points $F(z_1)$ and $F(z_2)$ the curve stays in the hemisphere $S_1$ because the geodesic curvature along the boundary curve does not change its sign as described earlier. Now we denote $\gamma$ an arc on $\partial A_1$ connecting $z_1$ and $z_2$ such that the length of $\gamma$ is less than or equal to $\pi$. Note that if $z_1$ and $z_2$ are antipodal points then there should be two choices of $\gamma$. Fix two points on $C_1$ in a way that two points are end points of a diameter of $C_1$ and a geodesic arc on a sphere connecting two fixed points contains $F(z_1)$ and $F(z_2)$. Now we can consider a variation of the great circle passing through two fixed points by rotating the circle along the diameter connecting two fixed points in a sense that as rotation starts great circles intersect $F(\gamma)$ transversally. Here two fixed points may be $F(z_1)$ and $F(z_2)$. 

As variation proceeds, we can see that there are cases where the last touching point with a great circle and the boundary curve $F(\gamma)$ exists on the hemisphere $S_1$ as denoted as $P$ in Figure 2 and Figure 3. However, the last touching point may not exist on the hemisphere as in Figure 4. We first look at the former case. In either case, the same argument can be adopted to show the contradiction. Accordingly, let us first consider when the last touching occurs while the great circle rotates with an angle less than $\pi$ as in Figure 2. Assume that the last touching point $P$ does exist in $S_1$ and is on a great circle $C_2$ as in Figure 2. As conormal vectors on $F(z_1)$ and $F(z_2)$ point opposite direction compared to the conormal vector defined on $P$, we see that geodesic curvatures have different sign at $F(z_i) (i=1,2)$ and $P$. But since the geodesic curvature of the curve cannot change its sign so we have a contradiction. 

Now we deal with the case where the last touching point does not exist in $S_1$ (see Figure 4). We denote $B_1$ a segment of the curve given by the immersion of $\gamma$ under $F$ and $B_2$ the complement of $B_1$ in the boundary curve  together with two points $F(z_i) (i=1,2)$. As we are assuming that we cannot find the last touching point as in the previous case, $B_1$ must meet $C_1$ at some point $q_1$ other than two points $F(z_1)$ and $F(z_2)$. Without loss of generality, we assume that the point $q_1$ is the closest one to $F(z_2)$. Consider a curve ${B_1}^\prime \subset B_1$ which joins $F(z_2)$ and $q_1$. Then we join from $q_1$ to $F(z_2)$ with the geodesic arc on the sphere to get a closed curve $D$. A region $S_D$ bounded by $D$ cannot be the whole hemisphere $S_1$ since $q_1$ is different from the point $F(z_1)$. Because of the embeddedness of the boundary curve, $B_2$ must either stays in $S_D$ or meets the geodesic arc connecting $F(z_2)$ and $q_1$ at some point $q_2$. The former case cannot happen as $B_2$ must contain the point $F(z_1)$. In the latter case let us denote the boundary curve connecting $F(z_2)$ and $q_2$ as ${B_2}^{\prime\prime }$. Since the curve  ${B_2}^{\prime\prime }$ is trapped in the domain $S_D$ which itself is not the whole hemisphere $S_1$, we easily see that we can rotate the great circle $C_1$ along some axis to find the last touching point with some great circle and the curve  ${B_2}^{\prime\prime }$. As in the previous case, we have a contradiction to the fact that the geodesic curvature of the boundary curve does not change sign. We have shown that the Gauss map of the surface is actually a one to one map on each of the boundaries of the annuls $A(1, R)$. 
\end{proof} 

\begin{figure}[H]
\centering
\begin{subfigure}[t]{6cm}
\centering
\scalebox{0.81}{
\begin{tikzpicture} 
 \shade[ball color = gray!40, opacity = 0.4] (0,0) circle (2cm);
      \draw[dashed] (0.5,1.4) .. controls (1,1.5) and (1.4,1.4)  .. (1.72,1);        
      \draw[dashed] (-0.5,1.4) .. controls (-0.3,1.3) and (0.3,1.3)  .. (0.5,1.4);  
      \draw[dashed] (-1.72,1) .. controls (-1.4,1.4) and (-1,1.5)  .. (-0.5,1.4);                     
      \draw (-1.72,1) .. controls (-1.3,1.3) and (-0.5,1.3)  .. (0,1);  
    \draw (0,1) .. controls (0.5,0.7) and (1.3,0.7)  .. (1.72,1);    
    \draw [-] (1.88,1.09)--(1.96,0.95) node[anchor=south] {};    
        \draw [-] (1.96,0.95)--(1.82,0.87) node[anchor=south] {};           
\draw[->] (0.903,0.525) -- (2.3, 1.335) node[anchor=west] {$F$};
\draw[->] (1.3,1.7) -- (2.1,0.42) node[anchor=west] {$N$};
\draw [->] (1.6,0.5)--(1.8,1.35) node[anchor=south] {$F_{\theta}$};
\end{tikzpicture}
}
\caption{Figure 1}
\end{subfigure}
\quad
\begin{subfigure}[t]{6cm}
\centering
\scalebox{0.81}{
\begin{tikzpicture}
 \shade[ball color = gray!40, opacity = 0.4] (0,0) circle (2cm);
  \draw (-2,0) arc (180:360:2 and 0.65) node[anchor=west] {$C_1$};
    \draw[color=red] (-2,0) arc (180:360:2 and 0.25);
  \draw[dashed] (2,0) arc (0:180:2 and 0.65);
    \draw[color=red, dashed] (2,0) arc (0:180:2 and 0.25);
      \filldraw[red] (-0.75,-0.6) circle (0pt) node[anchor=south] {$C_2$};
  \filldraw[black] (-0.75,-0.6) circle (0.8pt) node[anchor=north] {$F(z_1)$};
  \filldraw[black] (1.3,-0.51) circle (0.8pt) node[anchor=north] {$F(z_2)$};
  \filldraw[blue] (0.27,-0.3) circle (0pt) node[anchor=north] {$F(\partial A_1)$};
\draw[blue] (-1.92,0.55) .. controls (-1.6, -0.65) and (-0.72, -0.72) .. (-0.3,-0.5);      
\draw[blue]  (-0.3,-0.5) .. controls (0.25, -0.2) and (0.3, -0.2) .. (0.7,-0.4);  
\draw[blue]  (0.7,-0.4).. controls (1.25, -0.65) and (1.6, -0.55) .. (1.97,0.35);
  \filldraw[red] (0.3,-0.25) circle (0.8pt) node[anchor=south] {$P$};
\end{tikzpicture}
}
\caption{Figure 2}
\end{subfigure}
\quad
\begin{subfigure}[t]{6cm}
\centering
\scalebox{0.81}{
\begin{tikzpicture}
 \shade[ball color = gray!40, opacity = 0.4] (0,0) circle (2cm);
  \draw (-2,0) arc (180:360:2 and 0.65) node[anchor=west] {$C_1$};
    \draw[color=red] (-2,0) arc (180:360:2 and 0.25);
  \draw[dashed] (2,0) arc (0:180:2 and 0.65);
    \draw[color=red, dashed] (2,0) arc (0:180:2 and 0.25);
      \filldraw[red] (-0.75,-0.6) circle (0pt) node[anchor=south] {$C_2$};
  \filldraw[black] (-0.75,-0.6) circle (0.8pt) node[anchor=north] {$F(z_1)$};
  \filldraw[black] (1.3,-0.5) circle (0.8pt) node[anchor=north] {$F(z_2)$};
\draw[blue] (-1.92,0.55) .. controls (-1.14, -0.57) and (-0.72, -0.77) .. (-0.5,-0.5);   
\draw[blue]  (-0.5,-0.5) .. controls (-0.27, -0.153) and (-0.28, -0.153) .. (0,-0.7); 
      \filldraw[red] (-0.29,-0.25) circle (0.8pt) node[anchor=south] {$P$};
\draw[blue]  (0,-0.7).. controls (0.25, -1.1) and (0.35, -1.1) .. (0.6,-0.7);  
\draw[blue]  (0.6,-0.7) .. controls (0.8, -0.3) and (0.9, -0.37) .. (1.05,-0.44);  
\draw[blue]  (1.05,-0.44).. controls (1.28, -0.55) and (1.5, -0.57) .. (1.97,0.35);
\end{tikzpicture}
}
\caption{Figure 3}
\end{subfigure}
\quad
\begin{subfigure}[t]{6cm}
\centering
\scalebox{0.81}{
\begin{tikzpicture}
 \shade[ball color = gray!40, opacity = 0.4] (0,0) circle (2cm);
  \draw(-2,0) arc (180:360:2 and 0.6);
  \draw[dashed] (2,0) arc (0:180:2 and 0.6);
      \draw[color=red] (-2,0) arc (180:360:2 and 0.25);
          \draw[color=red, dashed] (2,0) arc (0:180:2 and 0.25);
 \filldraw (2,0) circle (0.8pt) node[anchor=north] {$F(z_2)$};
 \filldraw (-2,0) circle (0.8pt) node[anchor=north] {$F(z_1)$};
  \draw (-2,0) parabola (-1.1,1.5) node[anchor=west] {$B_1$};
 \draw (-1.1,1.5) .. controls (-1.05,1.65) and (-1,1.75) .. (-0.9,1.78);
 \draw [dashed] (-0.9,1.78) parabola (-0.3,0.5);
 \draw[dashed] (-0.3,0.5) ..controls  (-0.2, 0.2) and (0.2, 0.2) .. (0.3,0.5);
     \draw  (0.9,1.78).. controls  (1,1.75) and (1.05,1.65)  ..  (1.1,1.5);
         \draw (2,0) parabola (1.1,1.5);
\draw[dashed] (-2,0) .. controls (-1.95,0.4) and (-1.7, 0.6) ..(-1.5,0.8) node[anchor=east] {$B_2$};   
 \filldraw[red] (0,0.26) circle (0.8pt) node[anchor=south] {$$};
 \filldraw[green] (0.76,0.56) circle (0.8pt) node[anchor=south] {$q_2$};
 \filldraw[green] (0.33,0.6) circle (0.8pt) node[anchor=south] {$$};
  \filldraw[green] (0.2,0.6) circle (0pt) node[anchor=south] {$q_1$};
  \draw [blue, dashed] (0.9,1.78) parabola (0.3,0.48); 
     \draw[blue]  (0.9,1.78).. controls  (1,1.75) and (1.05,1.65)  ..  (1.1,1.5);
   \draw[blue] (2,0) parabola (1.1,1.5);
     \draw[blue, dashed] (2,0) arc (0:81:2 and 0.6);
\draw[dashed] (.5,0.3) .. controls (1.1, 1) and (1.8,1) .. (2,0);    
\end{tikzpicture}
}
\caption{Figure 4}
\end{subfigure}
\end{figure}

\begin{remark}\label{BC}
As the above theorem indicates, we can parametrize the surface with the inverse of the Gauss map. There are a few advantages to such a parametrization. First, the Weierstrass data now can be arranged as $(w,dh)$ where $w=g(z)$ and $dh$ is a holomorphic height differential expressed in a new coordinate $w$. It means that calculation with the data becomes much simpler dealing with the boundary conditions. As we now can view the Fraser-Li conjecture as determining whether the image of the Gauss map of the surface would be a concentric annulus in the complex plane or not, the holomorphic method exploiting the Weierstrass data may help to deal with the uniqueness of the surface. Also, there are interesting calculation results regarding the geodesic curvature of the boundary curve on the surface actually being constant. Indeed, one can find out calculating with the conformal coordinate $z$  given in the domain $A(1,R)$ that the geodesic curvature can be expressed in terms of the gauss map as follows:
\begin{align*}
\kappa_g=\frac{1}{|c|}\left[\frac{2}{1+|g|^2}\frac{1}{|g_\theta|}\mbox{Im}\left(\frac{g_{\theta\theta}}{g_\theta}-\frac{2|g|^2}{1+|g|^2}\frac{g_\theta}{g}\right)\right]|g_\theta|^2.
\end{align*}
We see that a term $|g_\theta|^2$ is dependent on the parameterization. However, after some calculation one can deduce that $\left[\frac{2}{1+|g|^2}\frac{1}{|g_\theta|}\mbox{Im}\left(\frac{g_{\theta \theta}}{g_\theta}-\frac{2|g|^2}{1+|g|^2} \frac{g_\theta}{g}\right)\right]$ is a term independent on the parameterization as it is indeed equal to $\lambda^2||\vec{\kappa}||$, where $\vec{\kappa}$ is a curvature vector of the Gauss image of a boundary curve on the unit sphere $\mathbb{S}^2$, and $\lambda$ is a conformal factor. The above calculation together with Proposition \ref{Thm1} shows that determining the shape of the Gauss map is actually a key to solve the Fraser-Li conjecture. Note that the critical catenoid has a property that $g_\theta$ is of constant modulus.
\end{remark}
\section{The Liouville equation and a free boundary minimal annulus}\label{Five}
Let $F: A(1, R)\to\mathbb{R}^3$ be a conformal harmonic immersion from $A(1, R)\coloneqq\{z\in\mathbb{C}\ |\ 1<|z|<R\}$, which represents an immersed free boundary minimal annulus $\Sigma$ in a ball $\mathbb{B}^3$. The theorem of Lewy \cite{Lewy} implies that $F$ can be extended as a minimal immersion defined in a slightly larger annulus $A(1-\epsilon, R+\epsilon)$ for some $\epsilon>0$. By abuse of notation, we also denote it as $F: A(1-\epsilon, R+\epsilon)\to \mathbb{R}^3$. 

Recall (\ref{secondf}) that the second fundamental form $\sigma$ is written as
\begin{align*}
\sigma=\mbox{Re}\left\{\tfrac{C_0}{z^2}dz^2\right\}
\end{align*}
for some real constant $C_0$. From this we compute
\begin{align*}
|\sigma|^2=\frac{2C_0^2}{r^4\lambda^4},
\end{align*}
where $\lambda$ is the conformal factor defined by $ds^2=\lambda^2|dz|^2$. Then Simons' identity $\Delta_{\Sigma}\log|\sigma|^2=-2|\sigma|^2$(cf. \cite{CM}, p. 71) gives
\begin{align}\label{simons}
	\frac{1}{\lambda^2}\Delta\log\frac{2C_0^2}{r^4\lambda^4}=-\frac{4C_0^2}{r^4\lambda^4}.
	\end{align}
Here we used $\Delta=4\frac{\partial}{\partial z}\frac{\partial}{\partial\bar{z}}$.

Now let $v\coloneqq\log\frac{1}{r^4\lambda^2}$. Since $\log r$ is a harmonic function, we deduce from (\ref{simons}) that $v$ satisfies the Liouville equation in $A(1-\epsilon, R+\epsilon)$:
	\begin{align*}
	\Delta v+2C_0^2e^v=0.
	\end{align*}
Moreover, as $\Sigma$ intersects $\partial\mathbb{B}^3$ orthogonally, the geodesic curvature of level curves $r=1$ and $r=R$ are equal to 1. This can be expressed as
	\begin{align*}
	-\frac{1}{r\lambda}\left(1+\frac{r}{\lambda}\frac{\partial \lambda}{\partial r}\right)&=1\ \mbox{if}\ r=1,\\
	-\frac{1}{r\lambda}\left(1+\frac{r}{\lambda}\frac{\partial \lambda}{\partial r}\right)&=-1\ \mbox{if}\ r=R.
	\end{align*}
Hence we obtain
	\begin{align*}
	\frac{\partial v}{\partial n}&=2e^{-\frac{1}{2}v}-2\ \mbox{if}\ r=1,\\
	\frac{\partial v}{\partial n}&=\frac{2}{R^2}e^{-\frac{1}{2}v}+\frac{2}{R}\ \mbox{if}\ r=R,
	\end{align*}
where $n$ is the inner unit normal to $A(1,R)$.

Combining all things above, we observe that $v$ gives rise to a solution of the following Liouville type boundary value problem:
\begin{align}\tag{E[$R$, $\epsilon$, $C_0$]}
\begin{cases}
\Delta v+2C_0^2e^v=0 &\mbox{in}\ A(1-\epsilon, R+\epsilon),\\
\frac{\partial v}{\partial n}=2e^{-\frac{1}{2}v}-2 &\mbox{if}\ |z|=1,\\
\frac{\partial v}{\partial n}=\frac{2}{R^2}e^{-\frac{1}{2}v}+\frac{2}{R} &\mbox{if}\ |z|=R.
\end{cases}
\end{align}

We now prove that the solution of E[$R$, $\epsilon$, $C_0$] also gives rise to a minimal surface orthogonal to spheres. Let $H_{R, \delta}\subset\mathbb{C}$ ($0\leq\delta<1$) be a horizontal slab given by
\begin{align*}
H_{R, \delta}=\{\xi\in\mathbb{C}\ |\ \log(1-\delta)<\mbox{Im}\xi<\log(R+\delta)\}
\end{align*}
and let $H^0_{R, \delta}\coloneqq H_{R, \delta}\cap\{\xi\in\mathbb{C}\ |\ 0\leq\mbox{Re}\xi\leq2\pi\}$. Note that $H_{R, \epsilon}$ is the universal cover of $A(1-\epsilon, R+\epsilon)$ with the covering map $z=e^{-i\xi}$, and $H^0_{R, \epsilon}$ is a fundamental domain.

\begin{theorem}\label{Liou}
Suppose that $v$ is a solution of \textup{E[$R$, $\epsilon$, $C_0$]}. There exist a minimal immersion $X: H_{R, \epsilon}\to\mathbb{R}^3$ and unit spheres $S_{O_1}$, $S_{O_2}$ (centered at $O_1$ and $O_2$, respectively) satisfying the following conditions:
\begin{itemize}
\item[(1)] For the conformal factor $\Lambda$ of $X$ given by $ds^2=\Lambda^2|d\xi|^2$, 
\begin{align*}
v(e^{-i\xi})=\log\frac{1}{\Lambda^2(\xi)|e^{-i\xi}|^2},\ \forall \xi\in H_{R, \epsilon}.
\end{align*}
\item[(2)] $S_{O_1}$ and $S_{O_2}$ intersect $X(H_{R, \epsilon})$ orthogonally along level curves $\textup{Im}\xi=0$ and $\textup{Im}\xi=\log R$, respectively.
\item[(3)] $X(H_{R, \epsilon})=\bigcup\limits_{n\in\mathbb{Z}}T^n\cdot X(H^0_{R, \epsilon})$ for some rigid motion $T$ in $\mathbb{R}^3$ such that $T\cdot X(H^0_{R, \epsilon})=X(H^0_{R, \epsilon}+2\pi)$.
\end{itemize}
We call $X(H^0_{R, 0})\subset X(H^0_{R, \epsilon})$ a fundamental piece.
\end{theorem}

\begin{proof}
We apply the same method as in \cite{Jimenez}. Consider the function $\tilde{v}$ defined by
\begin{align*}
\tilde{v}(\xi)=v(e^{-i\xi})+2\mbox{Im}\xi
\end{align*}
in the universal cover $H_{R, \epsilon}$. Then a simple computation shows that $\tilde{v}$ satisfies the following Liouville equation:
\begin{align}\label{liouuniv}
\begin{cases}
\Delta \tilde{v}+2C_0^2e^{\tilde{v}}=0 &\mbox{in}\ H_{R, \epsilon},\\
\frac{\partial \tilde{v}}{\partial n}=2e^{-\frac{1}{2}\tilde{v}} &\mbox{if}\ \mbox{Im}\xi=0\ \mbox{or} \log R.
\end{cases}
\end{align}
Here we used $\Delta=4\frac{\partial}{\partial \xi}\frac{\partial}{\partial \overline{\xi}}$, and $n$ is the inner unit normal to $H_{R, 0}$. Since $H_{R, \epsilon}$ is simply-connected, we observe that the solution $\tilde{v}$ is given by
\begin{align}\label{general}
\tilde{v}=\log \frac{4|h_{\xi}|^2}{(1+C_0^2|h|^2)^2}
\end{align}
for some locally univalent meromorphic function $h$ in $H_{R, \epsilon}$. 

Now we consider a minimal immersion $X: H_{R, \epsilon}\to\mathbb{R}^3$ defined by the Weierstrass data $(g, \omega)$:
\begin{align*}
g=C_0h,\ \omega=-\frac{1}{h_{\xi}}d\xi.
\end{align*}
As $h$ is locally univalent, it can only have simple poles. Therefore the Weierstrass data satisfies the condition concerning the order of zeros and poles mentioned in Section \ref{prelim}. Moreover, there are no period problems as $H_{R, \epsilon}$ is simply-connected, so we obtained a well-defined minimal immersion.

The induced metric is given by 
\begin{align*}
\Lambda^2|d\xi|^2=\frac{1}{4}|\omega|^2(1+|g|^2)^2=\frac{(1+C_0^2|h|^2)^2}{4|h_{\xi}|^2}|d\xi|^2,
\end{align*}
which implies that 
\begin{align*}
\Lambda^2=\frac{(1+C_0^2|h|^2)^2}{4|h_{\xi}|^2}.
\end{align*}
Hence we obtain
\begin{align*}
\log\frac{1}{\Lambda^2(\xi)|e^{-i\xi}|^2}=\log\frac{4|h_{\xi}|^2}{e^{2\text{Im}\xi}(1+C_0^2|h|^2)^2}=\tilde{v}(\xi)-2\mbox{Im}\xi=v(e^{-i\xi}),
\end{align*}
and this proves $(1)$.

On the other hand, it follows easily from the boundary condition in (\ref{liouuniv}) that the geodesic curvature of level curves $\mbox{Im}\xi=0$ and $\mbox{Im}\xi=\log R$ are equal to $1$. Moreover, the second fundamental form is
\begin{align}\label{SF}
\mbox{Re}\{dg\cdot\omega\}=\mbox{Re}\{-C_0d\xi^2\}
\end{align}
so that each level curve of $\mbox{Im}\xi$ is a line of curvature on the minimal surface. Then Proposition \ref{Thm1} implies that there exist unit spheres $S_{O_1}$, centered at $O_1$, and $S_{O_2}$, centered at $O_2$, such that $S_{O_1}$ and $S_{O_2}$ intersect $X(H_{R, \epsilon})$ orthogonally along level curves $\mbox{Im}\xi=0$ and $\mbox{Im}\xi=\log R$, respectively. Thus $(2)$ is obtained.

To prove $(3)$ we first observe that $\tilde{v}$ is $2\pi$-periodic:
\begin{align*}
\tilde{v}(\xi+2\pi)=v(e^{-i(\xi+2\pi)})+2\mbox{Im}(\xi+2\pi)=v(e^{-i\xi})+2\mbox{Im}\xi=\tilde{v}(\xi).
\end{align*}
This gives
\begin{align*}
\Lambda^2(\xi+2\pi)|d(\xi+2\pi)|^2=e^{-\tilde{v}(\xi+2\pi)}|d\xi|^2=e^{-\tilde{v}(\xi)}|d\xi|^2=\Lambda^2(\xi)|d\xi|^2,
\end{align*}
and we also have
\begin{align*}
\mbox{Re}\{-C_0d(\xi+2\pi)^2\}=\mbox{Re}\{-C_0d\xi^2\}.
\end{align*}
Therefore the first and second fundamental forms are $2\pi$-periodic. By the fundamental theorem of surfaces, $X(H^0_{R, \epsilon}+2n\pi)$ $\forall n\in\mathbb{Z}$ are congruent to $X(H^0_{R, \epsilon})$. If we denote by $T$ a rigid motion sending $X(H^0_{R, \epsilon})$ to $X(H^0_{R, \epsilon}+2\pi)$, i.e., $T\cdot X(H^0_{R, \epsilon})=X(H^0_{R, \epsilon}+2\pi)$, then one may deduce that 
\begin{align*}
X(H_{R, \epsilon})=\bigcup\limits_{n\in\mathbb{Z}}T^n\cdot X(H^0_{R, \epsilon}).
\end{align*} 
\end{proof}
We should mention that the minimal surfaces constructed in Theorem \ref{Liou} contain all immersed free boundary minimal annuli in a ball. It is clear from the argument at the beginning of this section with the aid of (\ref{SF}) and condition (1) since $d\xi=\frac{i}{z}dz$. Also, minimal annuli orthogonal to two spheres are contained in these examples.

Condition (3) implies that a minimal immersion obtained from the Liouville equation E[$R$, $\epsilon$, $C_0$] can be divided into pieces congruent to the fundamental piece. Therefore we have the following result concerning the structure of immersed free boundary minimal annuli:
\begin{cor}\label{struc}
Let $\Sigma$ be an immersed free boundary minimal annulus in $\mathbb{B}^3$. After rotating the surface, one can find $\Sigma_0\subseteq\Sigma$ such that 
\begin{align*}
\Sigma=\bigcup\limits_{n\in\mathbb{Z}}(Rot_{\frac{2k}{N}\pi})^n\cdot\Sigma_0
\end{align*}
for some $N\in\mathbb{Z}_{>0}$ and $0\leq k\leq N-1$, where $Rot_{\frac{2k}{N}\pi}$ means the rotation by $\frac{2k}{N}\pi$ with respect to the $x_3$-axis.
\end{cor}
\begin{proof}
By (3) in Theorem \ref{Liou}, there exist $\Sigma_0\subseteq\Sigma$ corresponds to the fundamental piece $X(H^0_{R, 0})$ and a rigid motion $T$ such that 
\begin{align*}
\Sigma=\bigcup\limits_{n\in\mathbb{Z}}T^n\cdot \Sigma_0.
\end{align*}
This implies that $\Sigma$ is invariant under $T$, and therefore $T$ must preserve the sphere $\partial \mathbb{B}^3$ since $\partial\Sigma\subset\bigcap\limits_{n\in\mathbb{Z}}T^n\cdot\partial\mathbb{B}^3$. Hence $T$ is an element of the special orthogonal group $SO(3)$.
First, assume that $\Sigma_0$ is simply-connected. To obtain an annulus by gluing congruent pieces $T^n\cdot\Sigma_0$, we should have $T^N=id$ for some $N$. Thus $T$ is a rotation by $\frac{2k}{N}\pi$ with respect to some axis. After rotating the surface, we may take $T=Rot_{\frac{2k}{N}\pi}$.
If $\Sigma_0$ is doubly-connected, then $T=id$, and we complete the proof.
\end{proof}
\begin{remark}\label{struc2}
For a minimal annulus orthogonal to two spheres, centered at $O_1$ and $O_2$, a similar result holds with $T$ a rotation with respect to $\overleftrightarrow{O_1O_2}$.
\end{remark}
The above corollary shows that studying fundamental pieces obtained from Liouville solutions will make us have a deeper intuition on free boundary minimal annuli in a ball. However, some Liouville solutions give rise to the case $O_1\neq O_2$. 

To reduce the case, we may impose the following condition to the solution $v$ of E[$R$, $\epsilon$, $C_0$]:
\begin{align}\label{AL}
2\int_0^{2\pi}\int_1^R\frac{1}{r^3}e^{-v}\textup{d} r\textup{d} \theta=\int_0^{2\pi}e^{-\frac{1}{2}v(1,\theta)}\textup{d} \theta+\int_0^{2\pi}\frac{1}{R}e^{-\frac{1}{2}v(R, \theta)}\textup{d} \theta.
\end{align}
This relation comes from the fact that every free boundary minimal surface $\Sigma$ in $\mathbb{B}^3$ satisfies $2|\Sigma|=|\partial\Sigma|$. 
\begin{prop}
Let $v$ be a solution of E[$R$, $\epsilon$, $C_0$] satisfying (\ref{AL}) and let $X: H_{R, \epsilon}\to\mathbb{R}^3$ be a minimal immersion obtained from $v$ as in Theorem \ref{Liou}. Denote by $\Sigma_0=X(H^0_{R, 0})$ a fundamental piece. If $O_1\neq O_2$, then
\begin{align}\label{fluxcond}
\int_{\Gamma_1}\nu_1ds=0,\ \int_{\Gamma_2}\nu_2ds=0,
\end{align} 
where $\nu_1$ and $\nu_2$ are outward unit conormal vectors to $\Gamma_1\coloneqq\overline{\Sigma}_0\cap S_{O_1}$ and $\Gamma_2\coloneqq\overline{\Sigma}_0\cap S_{O_2}$, respectively. 
\end{prop}
\begin{proof}
Let $\rho\coloneqq O_2-O_1$ and let $Y$ be the position vector with respect to $O_1$. We may write
$\partial\Sigma_0=\Gamma_1\cup\Gamma_2\cup C_1\cup C_2$, where $C_1, C_2=\emptyset$ if $\Sigma_0$ is doubly-connected. Then, by (\ref{AL}), we have
\begin{align}\label{Rt}
2|\Sigma_0|=|\Gamma_1|+|\Gamma_2|.
\end{align}

The divergence theorem and the minimality of $\Sigma_0$ imply that 
\begin{align}\label{Dot}
2|\Sigma_0|=\int_{\Sigma_0}\mbox{div}YdA=\int_{\Gamma_1}Y\cdot\nu_1 ds+\int_{\Gamma_2}Y\cdot\nu_2 ds+\int_{C_1\cup C_2}Y\cdot\nu ds.
\end{align}
Here we used $\nu$'s for outer unit conormals. Moreover, if we compute the torque with respect to $O_1$, then
\begin{align}\label{Wed}
0=\int_{\partial\Sigma_0}Y\wedge\nu ds=\int_{\Gamma_1}Y\wedge\nu_1 ds+\int_{\Gamma_2}Y\wedge\nu_2 ds+\int_{C_1\cup C_2}Y\wedge\nu ds.
\end{align}

On the other hand, by Corollary \ref{struc} and Remark \ref{struc2}, $X(H_{R, 0})$ is given by rotating $\Sigma_0$ and gluing them along $C_i$'s. This shows that the outer unit conormal to $C_1$ corresponds to the inner unit conormal to $C_2$ of the adjacent piece. Therefore we can conclude that
\begin{align}\label{Van}
\int_{C_1\cup C_2}Y\cdot\nu ds=0,\ \int_{C_1\cup C_2}Y\wedge\nu ds=0.
\end{align}

Since the surface and spheres intersect orthogonally, we have $Y=\nu_1$ on $\Gamma_1$ and $Y=\rho+\nu_2$ on $\Gamma_2$. Then (\ref{Rt}), (\ref{Dot}) and (\ref{Van}) give
\begin{align*}
\rho\cdot\int_{\Gamma_2}\nu_2 ds=0,
\end{align*}
and from (\ref{Wed}) and (\ref{Van}) we get
\begin{align*}
\rho\wedge\int_{\Gamma_2}\nu_2 ds=0.
\end{align*}
Thus we obtain $\int_{\Gamma_2}\nu_2 ds=0$ as $\rho\neq 0$. Similarly, $\int_{\Gamma_1}\nu_1 ds=0$ can be proved by using the position vector with respect to $O_2$, instead of $Y$.
\end{proof}  
Therefore we have proved that each Liouville solution satisfying (\ref{AL}) provides a (partially) free boundary fundamental piece $\Sigma_0$ with two possibilities: 
\begin{itemize}
\item $\Sigma_0$ in a ball ($O_1=O_2$)
\item $\Sigma_0$ orthogonal to two spheres ($O_1\neq O_2$) and satisfying (\ref{fluxcond})
\end{itemize}
Each of the cases provides natural research problems regarding the minimal annuli orthogonal to spheres as listed below. We first remark that the second case is not possible if $O_1$ and $O_2$ are far enough, or if the boundary curves on spheres are closed and embedded (in this case, (\ref{fluxcond}) cannot be achieved as each curve lies in a hemisphere). However, when $O_1$ and $O_2$ are sufficiently close to each other and the boundary is immersed, one can try to find an example corresponding to the second case:
\begin{que}
Can we find an example for the second case?
\end{que}

This example, if it exists, may degenerate when $O_1$ and $O_2$ are equal. Moreover, one can ask how many pieces that an immersed free boundary minimal annulus can have. Note that this question is closely related to symmetries of a surface. For example, if we assume that a minimal annulus is symmetric to two orthogonal vertical planes, then it should consist of an even number of fundamental pieces.
\begin{que}
If a free boundary minimal annulus orthogonal to spheres is not a part of the catenoid, can it be divided into more than one congruent piece?
\end{que}

We also recall that a rotationally symmetric free boundary minimal annulus in a ball is known to be the critical catenoid. On the other hand, a rotational symmetry of the Liouville solution implies that the metric has a rotational symmetry. Since the second fundamental form is also rotational in our case, the intrinsic rotational symmetry would lead to the extrinsic rotational symmetry on the surface. Thus, in order to prove the Fraser-Li conjecture or to find an immersed counter-example, we may study: 
\begin{que}
On what conditions the Liouville solution or the fundamental piece have rotational symmetry? 
\end{que}


\section*{Acknowledgements}
The authors would like to express their gratitude to Jaigyoung Choe for helpful comments concerning the application of the main theorem. This work was supported in part by NRF-2018R1A2B6004262.

\end{document}